\documentclass[12pt,a4paper]{amsart}
\usepackage{amsmath,amsfonts,amssymb,amsthm,amscd}

\input{comment.sty}
\includecomment{MM}
\includecomment{CL}

\renewcommand{\a}{\alpha}
\renewcommand{\b}{\beta}
\newcommand{\g}{\gamma}
\newcommand{\G}{\Gamma}
\renewcommand{\d}{\delta}

\renewcommand{\l}{\lambda}

\newcommand{\m}{\mu}
\newcommand{\n}{\nu}
\renewcommand{\o}{\omega}
\renewcommand{\O}{\Omega}
\renewcommand{\r}{\rho}
\newcommand{\s}{\sigma}

\renewcommand{\t}{\tau}
\renewcommand{\th}{\theta}
\newcommand{\e}{\varepsilon}

\newcommand{\x}{\xi}

\newcommand{\z}{\zeta}
\newcommand{\p}{\psi}

\newcommand{\Bc}{{\mathcal B}}

\newcommand{\C}{{\mathbb C}}
\newcommand{\E}{{\mathbb E}}
\newcommand{\Hyp}{{\mathbb H}}
\newcommand{\N}{{\mathbb N}}
\renewcommand{\P}{\mathbb P}
\newcommand{\R}{{\mathbb R}}

\newcommand{\Z}{{\mathbb Z}}

\newcommand{\Pb}{{\mathbb P}}
\newcommand{\pP}{{\mathcal P}}
\newcommand{\cF}{{\mathcal F}}

\newtheorem{theorem}{Theorem}[section]
\newtheorem{proposition}[theorem]{Proposition}
\newtheorem{lemma}[theorem]{Lemma}
\newtheorem{corollary}[theorem]{Corollary}

\newtheorem{fact}[theorem]{Fact}

\theoremstyle{definition}

\theoremstyle{remark}
\newtheorem{remark}[theorem]{Remark}

\newtheorem{question}[theorem]{Question}

\begin{document}

\title[Discrete random walks on the group Sol]
      {Discrete random walks on the group Sol}
\author{J\'{e}r\'{e}mie Brieussel, Ryokichi Tanaka}
\date{24th February 2014}

\maketitle

\begin{abstract}
The harmonic measure $\nu$ on the boundary of the group $Sol$ associated to a discrete random walk of law $\mu$ was described by Kaimanovich. We investigate when it is absolutely continuous or singular with respect to Lebesgue measure. By ratio entropy over speed, we show that any countable non-abelian subgroup admits a finite first moment non-degenerate $\mu$ with singular harmonic measure $\nu$. On the other hand, we prove that some random walks with finitely supported step distribution admit a regular harmonic measure. Finally, we construct some exceptional random walks with arbitrarily small speed but singular harmonic measures. The two later results are obtained by comparison with Bernoulli convolutions, using results of Erd\"os and Solomyak.
\end{abstract}

\section{Introduction}

Let $Sol$ denote the semi-direct product $\R \ltimes \R^2$ with action $z.(x,y)=(e^{-z}x,e^zy)$, endowed with the left-invariant Riemannian metric $ds^{2}=dz^{2}+e^{2z}dx^{2}+e^{-2z}dy^{2}$. The group $Sol$ is the simplest unimodular solvable non-nilpotent Lie groups. As a three dimensional manifold, it is one of the eight ``Thurston geometries". Even though it has both negative and positive sectional curvature, the visual boundary has been described by Troyanov \cite{T} as the union of two circles intersecting at two points. These two circles correspond to the boundaries of two hyperbolic planes of which $Sol$ is the horocyclic product. Measure-theoretically, this boundary is not different from the disjoint union of two real lines.

The aim of this article is to study the boundary behavior of discrete random walks, that is sequences of random variables $W_n=X_1\dots X_n$, where $X_i$ are independent group elements following a probability law $\m$ with discrete support and finite first moment. Kaimanovich has described the boundary behavior of such random walks on solvable Lie groups in \cite{Ksolv}. 

In the present particular case, this behavior depends firstly on the mean $\a=\E \m_z$ of the projection of $\m$ on the vertical $z$-axis. If $\a>0$, then the $x$-coordinate of the random walk converges almost surely to a real random variable $\x(\o)$. By adding a point at infinity, this real random variable can be seen as belonging to the boundary circle of the $zx$-hyperbolic plane. Similarly if $\a<0$, the $y$-coordinates converge to a real random variable $\x'(\o)$ viewed in the boundary of the $zy$-plane. The distribution $\n$ on $\R$ of the random variable $\x$ or $\x'$ is called the harmonic measure on the boundary of $Sol$. The distribution $\n$ is supported on one of these two real lines (or equivalently on one of the two boundary circles) according to the sign of the vertical drift $\a$.

When $\a \neq 0$ and the measure $\m$ is non-degenerate on a cocompact lattice, the measure space $(\R,\n)$ is actually the Poisson boundary of the random walk, as shown by Kaimanovich \cite{Ksolv}. If $\a=0$, the behavior of the random walk is not described in terms of the geometric boundary of $Sol$. In particular, when $\a=0$ and the measure $\m$ is supported on a cocompact lattice, the corresponding Poisson boundary is trivial.

A natural question is to determine the regularity of the harmonic measure $\n$ with respect to Lebesgue measure on $\R$. A measure is said non-degenerate (resp. non-degenerate on a group $\G$) if the semi-group generated by its support is in fact a group (resp. the group $\G$). We present the following results.

\begin{theorem}\label{main}
(1) For any countable non-abelian subgroup $\G$ of $Sol$, there exists a non-degenerate probability measure $\m$ on $\G$ such that the associated harmonic measure $\n$ is non-atomic and singular with respect to Lebesgue measure. 

(2) There exists a finitely supported non-degenerate probability measure $\m$ on $Sol$ such that  the harmonic measure $\n$ is absolutely continuous with respect to Lebesgue measure. In fact for any $k \in \N$, there exists such a $\m$ whose harmonic measure $\n$ admits a density function in the class $C^k$.

(3) For any $\a > 0$, there exists a finitely supported non-degenerate probability measure $\m$ on $Sol$ such that the associated random walk has speed $\a$ and the harmonic measure is non-atomic and singular with respect to Lebesgue measure.
\end{theorem}

The first part of Theorem \ref{main} is proved after Corollary \ref{fingen}. The two statements of the second part are Theorem \ref{thmsolomyak} and Theorem \ref{Ck}. The third part is stated more precisely as Theorem \ref{thmerdos}.

The measures $\m$ of Theorem \ref{main} (2) and (3) are very specific. In particular, their vertical components are supported on a lattice, and they satisfy an independence condition between their coordinate components. We are able to prove regularity by showing that these measures are related to the Bernoulli convolutions $b_\l$ with parameter $\l$ in $[\frac{1}{2},1[$, i.e. the laws of the real random variables $\sum_{j=0}^{\infty}x_{j}\l^{j}$, where $\{x_j\}_{j \in \N}$ is a sequence of independent variables equidistributed on $\{1,-1\}$. Absolute continuity and existence of $C^k$-densities of Bernoulli convolutions $b_\l$ for most values of the parameter in appropriate left-neighborhoods of $1$ were proved by Erd\"os \cite{E2} and Solomyak \cite{S}. 

On the other hand, for a parameter $\l$ that is the inverse of a Pisot number, Erd\"os proved that the Bernoulli convolution is singular with respect to Lebesgue \cite{E}. Theorem \ref{main} (3) is built on this result.

An interesting feature of the present work is to show the existence of a finitely supported probability measure on a Lie group, whose harmonic measure is absolutely continuous with respect to the natural measure on the boundary. Recall that the existence of measures with discrete support and finite first moment with absolutely continuous harmonic measures dates back to Furstenberg and the first rigidity results \cite{F} but their supports are a priori infinite. 

The question of finding finitely supported measures with absolutely continuous harmonic measure received a great deal of attention after the work of Kaimanovich and Le Prince \cite{KLeP} concerning discrete random walks on $SL(d,\R)$. For special linear groups, the harmonic measure is defined on the full flag manifold $\Bc=SL(d, \R)/P$, where $P$ is the parabolic subgroup of upper triangular matrices. They prove that any countable Zariski dense subgroup of $SL(d,\R)$ admits a symmetric non-degenerate probability measure whose harmonic measure is singular with respect to the natural Lebesgue measure on $\Bc$. Theorem \ref{main} (1) gives the same result (without symmetry) for $Sol$. It is proved by the same method as in \cite{KLeP}, estimating the Hausdorff dimension of $\n$ in terms of the ratio entropy by speed for the random walk.

Moreover, Kaimanovich and Le Prince conjectured that the harmonic measure on the flag space would be singular for any finitely supported non-degenerate symmetric random walk on $SL(d,\R)$. This conjecture was disproved by Bourgain \cite{Bo} who gave examples of finitely supported symmetric random walks on $SL(2,\R)$ with harmonic measures admitting a density in the class $C^k$ for arbitrary $k$ in $\N$. Theorem \ref{main} (2) is an analogue for $Sol$. Note however, that whereas the size of the support of the measure tends to infinity with the required $C^k$ regularity in Bourgain's exemples, the measures of Theorem \ref{main} (2) can all be chosen with a support of size 4 (see Theorems \ref{thmsolomyak} and \ref{Ck}).

Let us also mention that prior to Bourgain's examples,  B\'ar\'any, Pollicott and Simon constructed examples of finitely supported random walks on semi-groups of $SL(2,\R)$ with absolutely continuous harmonic measure \cite{BPS}. On the other hand, the statement of Kaimanovich-Le Prince conjecture turned out to be true for the Mapping Class Group of an orientable surface. There, the Poisson boundary is the space of projective measure foliations with hitting distribution \cite{KM}, and the harmonic measure is singular with respect to the natural Lebesgue measure class  provided the step distribution has finite support, as shown by Gadre \cite{G}.

In a discrete vs continuous dichotomy, the present random walks are related to the  Brownian motion associated with a Laplace operator with drift on the group $Sol$, studied by  Brofferio, Salvatori and Woess \cite{BSW}. The Brownian motion with vertical drift $\a$ on $Sol$ behaves like a random walk with vertical mean $\E \m_z=\a$, namely, it converges to the $x$ (resp. $y$) boundary line when the drift parameter $\a$ is positive (resp. negative). The induced distributions have similar descriptions as seen by comparing Proposition 4.2 in \cite{BSW} to Proposition \ref{xsi} below. Moreover, Brofferio, Salvatori and Woess describe an explicit $C^\infty$-density for the harmonic measure of the Brownian motion (Remark 4.3 in \cite{BSW}), whereas we obtain $C^k$-densities for larger $k$ by shrinking the size of the increments (see Theorem \ref{Ck}).

More generally, our work takes place in the growing subject of random processes on horocyclic products, such as the Diestel-Leader graphs, horocyclic products of regular trees, studied by Bertacchi \cite{B} and Kaimanovich and Woess \cite{KW}, and the treebolic space, the product of a hyperbolic plane and a tree, studied by Bendikov, Saloff-Coste, Salvatori and Woess \cite{Bend}. The group $Sol$ is the horocyclic product of two hyperbolic planes as a Riemannian manifold \cite{Wo}. In this view point, $Sol$ is realized in the product of two real affine groups which act on two hyperbolic planes respectively. In particular, since the boundary behavior of a random walk on $Sol$ is described by the boundary of one only of the two hyperbolic planes inside, according to the vertical drift, the same results as Theorem \ref{main} hold for the real affine group as well.

The organization of the paper is the following. In Section \ref{sol}, we describe the geometry of $Sol$  and in particular its visual boundary. Section \ref{poissonSol} is devoted to the behavior of random walks on $Sol$, and largely follows Kaimanovich's paper \cite{Ksolv}. Theorem \ref{main} (1) is proved in Section \ref{sing}. It follows from Theorem \ref{Cantor}, which asserts that the Hausdorff dimension of the harmonic measure is less than the ratio entropy over speed of the random walk. In Section \ref{abs}, we describe briefly Bernoulli convolution, with Erd\"os and Solomyak's Theorems and prove Theorem \ref{main} (2). Section \ref{Pisot} focuses on the particular case of ``Pisot" vertical lattices. We prove Theorem \ref{main} (3) and present a natural unanswered question about lattices of $Sol$. The paper ends with a short appendix devoted to classical facts about Pisot numbers and random walks on the integers.

{\it Landau asymptotic notation.}
Throughout the paper, we denote by $f(n)=o(g(n))$ if $\frac{f(n)}{g(n)} \to 0$ as $n \to \infty$, by $f(n)=O(g(n))$ if there exists some constant $C > 0$ independent of $n$ such that $f(n) \le Cg(n)$, and by $f(n) \sim g(n)$ if $\frac{f(n)}{g(n)} \to 1$.

\section{Description of the solvable Lie group $Sol$}\label{sol}

\subsection{Description as horocyclic product}

Recall that $Sol$ is the semi-direct product $\R \ltimes \R^2$. We denote its elements by coordinates $(z,x,y)$ in $\R^3$, where $z$ is considered a vertical and $x,y$ horizontal components, with product rule 
$$(z,x,y)(z',x',y')=(z+z',x+e^{-z}x',y+e^zy'),$$
and inverse $(z,x,y)^{-1}=(-z,-e^zx,-e^{-z}y)$. The group $Sol$ has a matrix representation of the form
$$
Sol=\left\{\left(\begin{array}{ccc} e^{-z} & x & 0 \\ 0 & 1 & 0 \\ 0 & y & e^{z} \end{array}\right) \Bigg| x, y,z \in \R \right\}.
$$
The Riemannian metric $ds^{2}=dz^{2}+e^{2z}dx^{2}+e^{-2z}dy^{2}$ is left-invariant. The $zx$-planes are totally geodesic hyperbolic planes in $Sol$. Indeed, the upper half-plane $\{(x,\x)|\x>0\}$ with metric $\frac{1}{\x^2}(d\x^2+dx^2)$ is turned into $\R^2$ with metric $dz^2+e^{2z}dx^2$ under the change of variable $z=-\log \x$. The later is called the logarithmic model of the hyperbolic plane. Similarly, the $zy$-planes are totally geodesic hyperbolic planes in $Sol$. In this case, the usual upper half-plane model is obtained by setting $z=\log \x$. Note that the $zx$ and $zy$-planes have ``upside-down" $z$-coordinate.

As a Riemannian manifold, $Sol$ is the horocyclic product of two hyperbolic planes, that is the hypersurface $\{z+z'=0\}$ of the direct product of a hyperbolic $zx$-plane with a hyperbolic $z'y$-plane, which is the 4-manifold $\Hyp^2 \times \Hyp^2$ homeomorphic to $\R^4$ with metric $ds^{2}=dz^{2}+e^{2z}dx^{2}+dz'^2+e^{2z'}dy^{2}$.

An important feature of the geometry of $Sol$ is that contrary to the $zx$ and $zy$-planes the horizontal $xy$-planes are very far from being totally geodesic. More precisely: 

\begin{lemma}\label{exp-distortion}
The horizontal plane $H:=\{(0, x, y) | (x, y) \in \R^{2}\}$ is embedded in $Sol$ with exponential distortion, i.e.,
$$
2\log \left(\frac{1}{4}\|(x, y)\|_{H}+\frac{1}{2}\right) \le d\left(id, (0, x,y)\right) \le 4\log \left(\| (x, y)\|_{H}+1\right),
$$
for any $(x, y) \in \R^{2}$, where $\|\cdot\|_{H}$ is the standard Euclidean norm in $H$.
\end{lemma}

\begin{proof}
A direct calculation using the hyperbolic metric on the $zx$-plane shows that
$$
d((0,0,0), (0,x,0))=2\log\left(\sqrt{1+\frac{|x|^{2}}{4}}+\frac{|x|}{2}\right),
$$
and the same equality holds for $(0,0,y)$.
These equalities yield the lemma.
\end{proof}

\subsection{The visual boundary of $Sol$}

A geodesic ray in a metric space $X$ is a function $\g:\R^+ \rightarrow X$ such that $d(\g(s),\g(t))=|t-s|$ for all $s,t \in \R^+$. Two geodesic rays $\g_1,\g_2$ are equivalent if their images are at bounded Hausdorff distance, i.e. if there exists $C$ such that $d(\g_1(t),\g_2(\R^+)) \leq C$ and $d(\g_2(t),\g_1(\R^+)) \leq C$ for all $t \in \R^+$. The visual boundary of the space $X$ with respect to the base point $x_0$ is the set of equivalence classes of geodesic rays starting at $x_0$.

The visual boundary of the hyperbolic plane is a circle $S^1$. In the $zx$-logarithmic model of $\Hyp^2$, any geodesic ray is equivalent to one of the vertical geodesics $\g_{x_0}(z)=(z,x_0)$ for $x_0 \in \R$ and $\g_\infty(z)=(-z,0)$ (and all the downwards vertical geodesics $\g(z)=(-z,x_0)$ are equivalent to $\g_\infty$). The visual boundary of $Sol$ was described by Troyanov.

\begin{theorem}[Troyanov \cite{T}]
The visual boundary $\partial Sol$ of $Sol$ with respect to a given point is the union of two circles that intersect at two points.
\end{theorem}

Troyanov proved this result by an explicit parametrization of all geodesics of $Sol$. An important observation is that the Riemannian geodesics not included in a $zx$-plane or a $zy$-plane are not rays, but only locally geodesic. Moreover, their $z$-coordinates are periodic in time.

These two circles consist of the visual boundaries of the $zx$-plane and the $zy$-plane. Following \cite{EFW2}, we call a geodesic ray of the form $\g^+_{x_0,y_0}(t)=(t,  x_{0}, y_{0})$ an {\it upward vertical geodesic ray}, and $\g^-_{x_0,y_0}(t)=(-t,  x_{0}, y_{0})$ a {\it downward vertical geodesic ray}. The equivalence classes of vertical geodesics are determined by
\begin{eqnarray*} \g^+_{x_0,y_0} \simeq \g^+_{x_1,y_1} &\Leftrightarrow & x_0=x_1, \\
\g^-_{x_0,y_0} \simeq \g^-_{x_1,y_1} &\Leftrightarrow & y_0=y_1. \end{eqnarray*}
The first circle can be parametrized by $\{\g^+_{x_0,0}|x_0 \in \R\} \cup \{\g^-_{0,0}\}$, corresponding to the boundary of the $zx$-plane, the second circle by $\{\g^-_{0,y_0}|y_0 \in \R\} \cup \{\g^+_{0,0}\}$, corresponding to $zy$-plane.

Since we will be interested in measure-theoretic properties on the boundary, we will describe it as the union of two disjoint real lines, parametrized respectively by the upward and downward vertical geodesics, rather than as a union of two circles:
\begin{eqnarray*}
\partial^+Sol &=& \{\g^+_{x_0,0}|x_0 \in \R\}, \\
\partial^-Sol &=& \{\g^-_{0,y_0}|y_0 \in \R\}.
\end{eqnarray*}
For further description of the group $Sol$, we refer to the literature, for instance \cite{T} for a Riemannian view point, \cite{EFW}, \cite{EFW2} for large scale geometry and rigidity, \cite{BSW} for the description of a compactification boundary. Note however that the visual boundary is only a subset of the compactification boundary of \cite{BSW}. Topologically, the visual boundary is the union of two circles intersecting at $\g_{0,0}^+$ and $\g_{0,0}^-$, because $\g^+_{x_0,0}\rightarrow \g^-_{0,0}$ as $|x_0| \rightarrow \infty$ and $\g^-_{0,y_0}\rightarrow \g^+_{0,0}$ as $|y_0| \rightarrow \infty$. On the other hand, the compactification boundary has the shape of a figure 8, because it contains an extra point $\d$ (not corresponding to a geodesic) such that $\g^+_{x_0,0}\rightarrow \d$ as $|x_0| \rightarrow \infty$ and $\g^-_{0,y_0}\rightarrow \d$ as $|y_0| \rightarrow \infty$. This difference has no importance measure-theoretically, so the results of the present paper apply to both boundaries.

\section{Boundary behavior of random walks}\label{poissonSol}

\subsection{Random walks on Sol}\label{RWonSol} Let $\m$ be a probability measure on $Sol$. We always assume that $\m$ has a finite first moment, i.e. satisfies $\int_{Sol} d(id,x)d\m(x) < \infty$ where $d$ is the left-invariant Riemannian metric on $Sol$.

We consider the random walk $\{W_{n}\}_{n=0}^{\infty}$ starting at $id$ of increment law $\m$, i.e. the sequence of random variables $W_{n}=X_{1}\cdots X_{n}$ and $W_{0}=id$, where $\{X_{j}\}_{j=0}^{\infty}$ is a sequence of independent random elements of law $\m$. With respect to the $\R^3$-coordinates, we write $X_j=(z_j, x_j,y_j)$ and $W_n=(S_n,U_n,V_n)$. Explicitely:
\begin{eqnarray}\label{prodlaw}
S_n &=& z_{1}+\cdots+z_{n}, \notag\\
U_n &=&  x_{1}+x_{2}e^{-S_{1}}+\cdots+x_{n}e^{-S_{n-1}}, \\
V_n &=& y_{1}+y_{2}e^{S_{1}}+ \cdots+y_{n}e^{S_{n-1}}. \notag
\end{eqnarray}
We denote by $\O$ the space of sample paths $\{W_{n}\}_{n=0}^{\infty}$ of random walks emanating from $id$, and by $\P$ the associated probability distribution on $\O$ given as the push-forward measure of $\m^{\times\infty}$ on $Sol^{\times \infty}$ under the map $\{X_{j}\}_{j=0}^{\infty} \mapsto \{W_{n}\}_{n=0}^{\infty}$.

Let $\m_{z}$ (respectively $\m_x$, $\m_y$) be the image measure under the projection to the $z$-component $Sol \to \R$, $(z, x, y) \mapsto z$ (respectively to the $x$ and $y$-components). As $\m$ has a finite first moment, the measure $\m_{z}$ admits a finite mean $\a=\E\m_{z}$ which plays a crucial role in the boundary behavior, since it is the vertical drift of the random walk (see Theorem \ref{LLN} and Corollary \ref{speed}). 

Let $\pi$ denote the projection $(z,x,y) \mapsto (z,x)$ of $Sol$ onto the $zx$-plane. We say that the measure $\m$ gives independance to the $z$ and $x$-components if the projected measure on the $zx$-plane is the product of the image measures on the $z$ and $x$-components, i.e. $\pi_\ast \m=\m_z \times \m_x$.

The assumption of finite first moment of $\m$ gives a geometric control on the behavior of the sequence $\{X_{j}\}_{j=0}^{\infty}$. In particular:

\begin{lemma}\label{ffm}
For an i.i.d. sequence $\{X_{j}=(z_{j},x_{j}, y_{j})\}_{j=1}^{\infty}$ of law $\mu$ with finite first moment, we have $|x_{j}|, |y_{j}|=e^{o(j)}$, $\P$-a.s.
\end{lemma}

\begin{proof}
The law of large numbers shows the convergence $\P$-a.s. of 
$$\lim_{n \to \infty}\frac{1}{n}\sum_{j=1}^n d(id,X_j) = \int_{Sol} d(id,x)d\mu(x)<\infty,$$ 
thus $d(id,X_j)=o(j)$ and by Lemma \ref{exp-distortion}, we get  $\log|x_j|=o(j)$ and $\log|y_j| =o(j)$, $\P$-a.s..
\end{proof}

\subsection{Hitting distribution on the boundary}\label{Poissonboundary} In this section, we describe the boundary behavior of a random walk with finite first moment on $Sol$. It is a particular case of results by Kaimanovich \cite {Ksolv} applying to certain class of solvable Lie groups. For completeness, we give short proofs in our simplified setting. The description below should be compared with the boundary behavior of the Brownian motion studied in \cite{BSW}.

Recall that the random walk at the time $n$, $W_n=(S_n, U_n, V_n)$ has the form (\ref{prodlaw}) by the product law.
By the law of large numbers, the vertical drift $\a=\E\m_z$ determines $\P$-a.s. the behavior of the vertical component $\lim \frac{1}{n}S_n=\a$, and thus which part $\partial^+Sol$ or $\partial^-Sol$ of the boundary is hit by the random walk. The asymptotic behavior of horizontal components $U_n$ (resp. $V_n$) is described by its almost sure limit $\x$ (resp. $\x'$) which has the form of a random infinite series. The limits $\x$ and $\x'$ will be regarded as the hitting points of random walks on the boundary.

\begin{proposition}\label{xsi}
If $\a>0$, the sequence $(U_n)$ converges $\P$-a.s. to $\x=\sum_{j=1}^{\infty}x_{j}e^{-S_{j-1}}$. \newline If $\a<0$, the sequence $(V_n)$ converges $\P$-a.s. to $\x'=\sum_{j=1}^{\infty}y_{j}e^{S_{j-1}}$.
\end{proposition}

\begin{proof}
For $\a>0$ and using Lemma \ref{ffm}, for any small $\e >0$,
there exists $\P$-a.s. an integer $N$, depending on the sample, such that for any $n > N$,
$|S_{n}/n - \a| \le \e$, $|x_{n}|, |y_{n}| \le e^{\e n}$.
Thus, we have $e^{-S_{n-1}}|x_{n}| \le e^{-(\a-2\e)n}e^{\a-\e}$ and the series converges.
\end{proof}

Denote by $\n$ (respectively $\n'$) the distribution of the random variable $\x$ (resp. $\x'$) on $\R$, i.e. for any Borel set $A$, set $\n(A)=\P(\x \in A)$. We can now describe the hitting distribution on the boundary of $Sol$, which we call the harmonic measure on the boundary associated to the random walk of increment law $\m$.

\begin{theorem}\label{boundary}
If $\a>0$, the harmonic measure is given by the measure $\n$ on $\partial^+Sol = \R$ and gives no mass to $\partial^-Sol$. If $\a<0$, the harmonic measure is given by the measure $\n'$ on $\partial^-Sol = \R$ and gives no mass to $\partial^+Sol$.
\end{theorem}

By abuse of langage, we call the measures $\n$ or $\n'$ on $\R$ the harmonic measure of the random walk. This theorem follows from the global law of large numbers on solvable Lie groups stated as Theorem 4.2 in \cite{Ksolv}. In our setting:

\begin{theorem}[Kaimanovich \cite{Ksolv}]\label{LLN}
If $\a>0$, set $g=(0, \x, 0)(\a, 0,0)(0, -\x, 0)$. If $\a<0$, set $g=(0, 0, \x')(\a, 0,0)(0, 0, -\x')$. If $\a=0$, set $g=id$. In all cases:
$$
\lim_{n \to \infty}\frac{1}{n}d(g^{n}, W_{n}) = 0, \ \ \ \ \P\text{-a.s.}
$$
\end{theorem}

Theorem \ref{LLN} implies Theorem \ref{boundary}. Indeed if $\a>0$, consider the Borel measurable map from $\O$ to $\partial^+Sol$ given by $\o \mapsto \g^+_{\x(\o),0}$. Since $d(g^n,\g^+_{\x(\o),0}(n \a))=d(id,(0,\x,0))<\infty$, we have $\P$-a.s. that $d(W_n,\g^+_{\x(\o),0}(n \a))=o(n)$, showing that the random walk ${W_n}$ behaves asymptotically as the vertical geodesic $\g^+_{\x,0}$ (whereas non equivalent vertical geodesics diverge linearly $d(\g^+_{x_0,0}(t),\g^+_{x_1,0}(t)) \sim t$ when $x_0 \neq x_1$). The case $\a<0$ is similar.

\begin{proof}[Proof of Theorem \ref{LLN}]
By Proposition \ref{xsi}, the map $\o \mapsto g(\o)$ from the space of sample paths $\O$ to $Sol$ is measurable. We treat the case $\a>0$. We have 
\begin{align*}
d(g^{n}, W_{n})  &\le d(id,(-n\a+S_n,0,0)) \\
&+d((-n\a+S_n,0,0),(-n\a+S_n,\x-e^{n\a}(\x-U_{n}),e^{-n\a}V_{n})).
\end{align*}
The first term of the right hand side equals to $|-n\a + S_{n}|$, which is $o(n)$, $\P$-a.s.
By left-invariance of the metric, the second one is equal to $d(id, (0, e^{-n\a + S_{n}}(\x-e^{n\a}(\x-U_{n})), e^{-S_{n}}V_{n}))$,
which is bounded from above by 
$$
4\log \left(\left|e^{-n\a+S_{n}}\x-e^{S_{n}}\sum_{j=n+1}^{\infty}x_{j}e^{-S_{j-1}}\right|+1\right) + 4\log \left(\left|e^{-S_{n}}\sum_{j=1}^{n}y_{j}e^{S_{j-1}}\right|+1\right),
$$
using Lemma \ref{exp-distortion} and the inequality $\log(|x| +|y|+1) \le \log(|x|+1)+\log(|y|+1)$.
For any $n \ge N$, the first term and the second one are less than or equal to 
$2\log\left(|\x|e^{\e n}+ Ce^{3\e n}+1\right)$
and
$2\log\left( Ce^{3\e n}+1\right)$, respectively for some constant $C>0$ which does not depend on $n$.
This shows that $d(g^{n}, W_{n})=o(n)$, $\P$-a.s.
\end{proof}

\begin{remark}
If the probability measure $\m$ with finite first moment is supported on a cocompact lattice $\G$ of the group $Sol$, then the boundary $\partial Sol$ endowed with the harmonic measure is the Poisson boundary of $(\G,\m)$. If $\a=0$, then the Poisson boundary on the cocompact lattice $\G$ with measure $\m$ is trivial. This follows from Kaimanovich ray approximation argument (see Theorem 5.5 in \cite{Khyperbolic} and Theorem 4.3 in \cite{Ksolv}).
\end{remark}

Theorem \ref{LLN} permits to compute the speed of the random walk.

\begin{corollary}\label{speed}
The speed of the random walk of law $\m$ on the group $Sol$ is the absolute value of the vertical drift $\a=\E \m_z$, that is
$$
\lim_{n \to \infty}\frac{1}{n}d(id, W_{n})=|\a|, \ \P\text{-a.s.}
$$
\end{corollary}

\begin{proof}
We treat the case $\a>0$. By Lemma \ref{exp-distortion} and the triangular inequality,
$n \a \le d(id, g^{n}) \le n \a + 4\log(\|(\x-e^{-n\a}\x, 0)\|_{H}+1)$, thus $d(id, g^{n}) \sim \a n$, and by Theorem \ref{LLN}, $d(id, W_{n}) \sim \a n$.
\end{proof}

\subsection{Elementary properties of the harmonic measure}

The group $Sol$ acts affinely on the boundary by $g.\x=x+e^{-z}\x$ for $\x \in \partial^+Sol$ and $g.\x' = y+e^z\x'$ for $\x' \in \partial^-Sol$, with notation $g=(z,x,y)$. By symmetry in Theorem \ref{boundary}, we restrict our considerations to the case $\a>0$ and the subset $\R=\partial^+Sol$ of the boundary. The measure $\n$ is pushed by action of $g$ to the measure $g\n$ given by $g\n(A)= \n(g^{-1}.A)=\P(g.\x \in A)$. 

The harmonic measure $\n$ is $\m$-stationary, i.e. satisfies $\n=\int_{g \in Sol}g\n d\m(g)$.
In fact, the harmonic measure $\n$ is the unique such probability measure on $\R=\partial^+Sol$. It also has a law of pure type, i.e. either absolutely continuous or completely singular with respect to Lebesgue measure.

\begin{proposition}\label{unique-harm}
Assume that $\m$ satisfies $\a=\E \m_{z}>0$.
Then the harmonic measure $\n$ is the unique $\m$-stationary probability measure on $\R=\partial^+Sol$.
Moreover, the harmonic measure $\n$ is either absolutely continuous or completely singular with respect to Lebesgue measure.
\end{proposition}

\begin{proof}
Suppose that $\l$ is a $\m$-stationary probability measure on $\R$.
First, we will show the uniqueness; $\l=\n$.
Note that a sequence of measures $W_{n}\l$ on $\R$ converges to the point measure $\d_{\x}$ weakly, $\P$-a.s.,
since for any bounded continuous function $f : \R \to \R$,
$$
\int_{\R}f(x)dW_{n}\l(x)=\int_{\R}f(x_{1}+x_{2}e^{-S_{1}}+\cdots + x_{n}e^{-S_{n-1}}+e^{-S_{n}}x)d\l(x) \to f(\x),
$$
$\P$-a.s., by the Lebesgue dominated convergence theorem.
Here $\x=\sum_{j=1}^{\infty}x_{j}e^{-S_{j-1}}$, $\P$-a.s.
Now $\l$ is $\m$-stationary; by induction, for any $n$,
$
\l=\int_{Sol}g\l d\m^{\ast n}(g)=\int_{\O}W_{n}\l d\P.
$
Since $W_{n}\l$ converges to $\d_{\x}$ weakly, $\P$-a.s., again by the Lebesgue dominated convergence theorem,
$\int_{\O}W_{n}\l d\P$ converges to $\int_{\O}\d_{\x}d\P=\n$ weakly, and thus $\l=\n$.

Next, to prove that $\n$ is either absolutely continuous or singular with respect to Lebesgue measure, take the Lebesgue decomposition $\n=\n_{ac} + \n_{s}$, where $\n_{ac}$ (resp. $\n_{s}$) is the absolutely continuous (resp. singular) part.
Also take the Lebesgue decomposition $g\n=(g\n)_{ac}+(g\n)_{s}$ for $g$ in $Sol$.
Here we have $(g\n)_{ac}+(g\n)_{s}=g\n_{ac}+g\n_{s}$.
Since for any Lebesgue measure $0$ set $C$ and for any $g$ in $Sol$, $g.C$ has also the Lebesgue measure $0$,
$(g\n)_{ac}$, $g\n_{ac}$ are absolutely continuous with respect to Lebesgue measure and $(g\n)_{s}$, $g\n_{s}$ are mutually singular with Lebesgue measure.
Therefore, $(g\n)_{ac}=g\n_{ac}$ and $(g\n)_{s}=g\n_{s}$.
Now $\n$ is $\m$-stationary,
$\n=\int_{Sol}g\n d\m(g)$, and it follows that 
$\n_{ac}=\int_{Sol}(g\n)_{ac}d\m(g)$ and
$\n_{s}=\int_{Sol}(g\n)_{s}d\m(g)$.
The equalities $(g\n)_{ac}=g\n_{ac}$ and $(g\n)_{s}=g\n_{s}$ imply that $\n_{ac}$ and $\n_{s}$ are $\m$-stationary.
By the uniqueness of $\m$-stationary probability measure, if neither $\n_{ac}$ nor $\n_{s}$ is $0$,
then $\n_{ac}$ and $\n_{s}$ are equal up to normalization.
This contradicts.
Hence either $\n_{ac}$ or $\n_{s}$ has to be $0$.
\end{proof}

We focus on the case where $\m$ has countable support, generating (as a semi-group) a countable subgroup $\G$ of $Sol$. The inherited action of $\G$ on $\R=\partial^+Sol$ is non-elementary if each orbit contains infinitely many points, i.e. for each $p \in \R$, $\#\{g.p \ | \ g \in \G\}=\infty$.

\begin{proposition}\label{non-atomic}
If the countable subgroup $\G$ of $Sol$ generated by $Supp(\m)$ acts non-elementary on $\R$, then the harmonic measure $\n$ is non-atomic.
\end{proposition}

\begin{proof}
Suppose that $\n$ has atoms on $\R$. Let $m>0$ be the maximal weight of atoms and $p$ a point which has $\n(p)=m$. Since $\n$ is $\m$-stationary, $m=\n(p)=\sum_{g}\m(g)g\n(p)$, hence for $g \in Supp(\m)$, $\n(g^{-1}.p)=m$, and by induction, for any $g \in \G$, $\n(g^{-1}.p)=m$.
This is impossible if we assume that $\G$ acts on $\R$ non-elementary.
\end{proof}

We give a simple criterion for a countable group to have a non-elementary action at the boundary. For $g=(z,x,y)$ in $Sol$ with $z \neq 0$, set:
$$p^+(g)=\frac{x}{1-e^{-z}} \textrm{ and } p^-(g)=\frac{y}{1-e^{z}}. $$

\begin{fact}\label{pp}
If $p^+(g') \neq p^+(g)$ (resp. $p^-(g') \neq p^-(g)$), the group generated by $g,g'$ and their inverses has a non-elementary action on $\partial^+Sol$ (resp. $\partial^-Sol$).
\end{fact} 

\begin{proof}
The affine action on $\partial^+Sol$ of powers of $g$ is given by $g^n(\x)=e^{-nz}(\x+\frac{x}{e^{-z}-1})+\frac{x}{1-e^{-z}}$, so the $g$-orbit of any point $p \in \R \setminus \{p^+(g)\}$ is infinite. The hypothesis ensures that $p^+(g)$ has an infinite $g'$-orbit.
\end{proof}

Here is a criterion ensuring that the harmonic measure has full support in the boundary. Recall from Section \ref{RWonSol} that the measure $\m$ is said to give independance to the $z$ and $x$-components if its projection onto the $zx$-plane is a product measure $\pi_\ast \m=\m_z \times \m_x$ .

\begin{proposition}
Assume that $\m$ gives independence to the $z$ and $x$-components and that $\a=\E\m_z>0$,  $\m_x(\R_\ast^+)>0$, $\m_x(\R_\ast^-)>0$, $\m_z(\R_\ast^+)>0$, $\m_z(\R_\ast^-)>0$, $\m_z(0)>0$, then the group $\G$ generated by $Supp(\m)$ acts non-ementarily on $\partial^+Sol=\R$ and the harmonic measure has full support $Supp(\n)=\R$.
\end{proposition}

\begin{proof}
Since $\n(\R)=1$, there exists $t$ such that $\n([-t,t])>0$. Given $\x \in \R$ and $\e>0$, set $z \in \R$ such that $e^{-z}t\leq \e$, and choose $X_1,\dots, X_k$ in $Supp(\m)$ such that $S_k=z(X_1\dots X_k)\geq z$. 
Then choose $X_{k+1},\dots ,X_n$ such that $z_{k+1}=\dots=z_n=0$ and $|\x-U_n|\leq \e$, where $U_n=\sum_{j=1}^n x_je^{-(z_1+\dots+z_{j-1})}$.
Now if $\x_\infty=\x(X_{n+1}X_{n+2}\dots)$ has law $\n$, then $\x(X_1\dots X_nX_{n+1}\dots)=U_n+e^{-S_k}\x_\infty$ belongs to $[x-2\e,x+2\e]$ as soon as $\x_\infty\in [-t,t]$. This shows $\n([x-2\e,x+2\e])\geq \P(X_1\dots X_n) \n([-t,t])>0$. In particular, the orbit of any $\x_\infty$ is dense in $\R$.
\end{proof}

\section{Singular harmonic measures}\label{sing}

In this section, we provide an upper bound for the Hausdorff dimension of the harmonic measure in terms of entropy and speed of the random walk. We deduce a sufficient criterion for singularity of this measure, and deduce that any countable subgroup of $Sol$ admits a probability measure $\m$ such that the associated random walk has a harmonic measure singular with respect to Lebesgue measure on the boundary.

\subsection{Upper bound on the Hausdorff dimension}

The Hausdorff dimension of a measure $\n$ on $\R$ is defined by 
$$
\dim \n:=\inf\{\dim_{H}A \ | \ A \subset \R, \n(A^{\mathsf C})=0\},
$$
where $\dim_{H}A$ denotes the Hausdorff dimension of the set $A$.

The Shannon entropy of the countable supported probability measure $\m$ is the quantity $H(\m):=-\sum_{s}\m(s)\log\m(s)$. When it is finite, the entropy of the random walk of law $\m$ is
$$
h_\m=\lim_{k \rightarrow \infty} \frac{H(\m^{\ast k})}{k}=\inf_k  \frac{H(\m^{\ast k})}{k}.
$$
The entropy of the random walk measures how fast the convolution measures $\m^{\ast k}$ diffuse the mass in the countable group (see for instance \cite{KV} about entropy on countable groups).

The speed of the random walk is the quantity $\lim_{n \to \infty}\E d(id,W_n)/n$ measures the expected value of the distance in $Sol$ between the random walk and its starting point. By Corollary \ref{speed}, the speed of the random walk of law $\m$ on the group $Sol$ is equal to $|\a|=|\E \m_z|$.

The Hausdorff dimension of the harmonic measure can be bounded from above in terms of entropy and speed by the:

\begin{theorem}\label{Cantor}
Let $\m$ be a countably supported probability measure on $Sol$ with finite first moment and finite entropy and such that $\a=\E \m_z \neq 0$. Let $\n$ be the harmonic measure on $\partial Sol$ corresponding to the pair $(\G,\m)$. Then
$$
\dim \n \le \frac{h_\m}{|\a|}.
$$
\end{theorem}

Estimation of dimension in terms of entropy and speed is classical (see for instance \cite{LSL2} for discrete subgroups of $SL(2,\C)$, \cite{KLeP} for $SL(d,\R)$ and the expository introduction of \cite{KLeP} for more information about when the above inequality holds). In the case of free groups, the Hausdorff dimension of the harmonic measure is precisely equal to the asymptotic entropy divided by the speed (see Ledrappier \cite{L}). However the inequality of Theorem \ref{Cantor} has to be strict in the case $h_\m/|\a|>1$, since the Hausdorff dimension of $\n$ cannot exceed the dimension $1$ of the boundary $\partial Sol$.

In order to estimate the dimension of $\n$,  we use the following lemma (see \cite{P}, Theorem 7.1. Chapter 2, pp. 42):

\begin{lemma}[Frostman]\label{Frostman}
Let $B_{r}(x)$ denotes the ball of radius $r$ centered at $x$. If for $\n$-a.e. $x$, the inequalities
$$
\d_{1} \le \liminf_{r \to 0}\frac{\log \n(B_{r}(x))}{\log r} \le \d_{2},
$$
hold, then $\d_{1} \le \dim \n \le \d_{2}$.
\end{lemma}

\begin{proof}[Proof of Theorem \ref{Cantor}] We treat the case $\a>0$.

Given two integers $n$ and $k$, we say two trajectories $\o$ and $\o'$ in $\O$ are equivalent if  $X_{ik+1}X_{ik+2}\dots X_{(i+1)k}(\o)=X_{ik+1}X_{ik+2}\dots X_{(i+1)k}(\o')$ for all $0 \le i \le n$, that is if the $(n+1)$ first $k$-step increments of the random walk coincide.
This equivalence relation defines a partition $\pP_{n}^k$ of $\O$.
Note that for $m > n$, the partition $\pP^k_{m}$ is a refinement of $\pP_{n}^k$.

Recall that $S_n$ is the vertical component of the random walk at time $n$. Given $\e>0$ and an integer $N$, define the set 
$$A_{\e}^{N}:=\{ \o \in \O \ |  \forall n \ge N, S_{n}(\o)/n \ge \a - \e \textrm{ and } |x_n| \leq e^{\e n}\}.$$
The sequence of sets $\{A_{\e}^{N}\}_{N}$ is increasing, that is, $A_{\e}^{N} \subset A_{\e}^{N+1}$ for any $N$.
By Lemma \ref{ffm} and as the law of large numbers implies $\lim_{n \to \infty}S_{n}/n = \a$ almost surely, we have $\Pb(\cup_{N}A_{\e}^{N})=1$.
Therefore, for any $\e>0$, there exists $N(\e)$ such that $\Pb(A_{\e}^{N(\e)}) \ge 1- \e$.
We write $A$ for $A_{\e}^{N(\e)}$ and take $(n+1)k > N(\e)$.

With the notations of Section \ref{RWonSol} and by Proposition \ref{xsi}, we have $\x=U_{(n+1)k}+\sum_{j>(n+1)k} x_j e^{-S_{j-1}}$. 
For $\o \in A$, this ensures
$$
|\x(\o)-U_{(n+1)k}(\o)| \leq \sum_{j>(n+1)k}e^{j\e}e^{-j(\a-\e)} \leq  Ce^{-nk(\a-2\e)}
$$ 
for some constant $C>0$.

Let $\pP_{n}^k(\o)$ be the element of the partition $\pP_{n}^k$ containing $\o$. For any $\o' \in \pP_{n}^k(\o)$, we have $U_{(n+1)k}(\o)=U_{(n+1)k}(\o')$, and so for any $\o' \in A \cap \pP_{n}^k(\o)$, we have $|\x(\o) - \x(\o')| \le 2Ce^{-nk(\a-2\e)} $. Finally, for any $\o \in A$
\begin{eqnarray}\label{nuA}\Pb(A\cap \pP_{n}^k(\o)) \le \Pb(\{\o' \in \O \ | \ |\x(\o)- \x(\o')| \le r_{n}^k\})=\n(B_{r_n^k}(\x(\o))).\end{eqnarray}

Define $r_{n}^k:=2Ce^{-nk(\a-2\e)}$, then $r_{n}^k \to 0$ as $ n \to \infty$.
By the Frostman Lemma \ref{Frostman}, our proof is reduced to give an upper bound for 
$$
\liminf_{n \to \infty}\frac{\log \n(B_{r_{n}^k}(x))}{\log r_{n}^k}, \ \ \ \  \text{$\n$-a.e. $x$}.
$$

For $r_n^k <1$, inequality (\ref{nuA}) implies that
$$
\frac{\log \n(B_{r_n^k}(\x(\o)))}{\log r_{n}^k} \le \frac{\log \Pb(A\cap \pP^k_{n}(\o))}{\log (2Ce^{-nk(\a-2\e)})}.
$$

Now the martingale convergence theorem ensures that for $\Pb$-a.e. $\o$, 
$$
\frac{\Pb(A\cap \pP^k_{n}(\o))}{\Pb(\pP^k_{n}(\o))} \underset{n \to \infty}{\longrightarrow} \Pb(A | \cF_{\infty})(\o), 
$$
where $\cF_{\infty}$ is the $\s$-algebra generated by $\{ X_{ik+1}X_{ik+2}\dots X_{(i+1)k}, i\geq 0\}$. Note that $\cF_{\infty}$ is a sub $\s$-algebra of the standard one generated by the cylinder sets in $\O$.
Here $\Pb(A | \cF_{\infty})(\o) > 0$ for $\Pb$-a.e. $\o \in A$, so we obtain
\begin{align}\label{limitinf}
\liminf_{n \to \infty}\frac{\log \Pb(A\cap \pP^k_{n}(\o))}{-nk(\a-2\e)} = \liminf_{n \to \infty} \frac{\log \Pb(\pP^k_{n}(\o))}{-nk(\a-2\e)}. 
\end{align}

On the other hand, by definition of the partition $\pP_n^k$, we have
$$
\P(\pP^k_{n}(\o))=\m^{\ast k}(X_{1}X_2\dots X_k(\o))\cdots \m^{\ast k}(X_{nk+1}X_{nk+2}\dots X_{(n+1)k}(\o)),
$$
and as $-\E\log\m^{\ast k}(X_{1}X_2\dots X_k)=H(\m^{\ast k}) < \infty$,
by the strong law of large numbers, we get
$$
-\frac{1}{n}\log\Pb(\pP^k_{n}(\o)) \underset{n \to \infty}{\longrightarrow} H(\m^{\ast k}), 
$$
$\Pb$-a.e. $\o \in \O$.
Thus the right hand side of (\ref{limitinf}) is bounded from above by $H(\m^{\ast k})/k(\a-2\e)$, $\Pb$-a.e.
Here $\e >0$ is an arbitrary positive value, and $k$ is an arbitrary integer, so
$$
\liminf_{n \to \infty}\frac{\log \n(B_{r_{n}}(x))}{\log r_{n}} \le \frac{h_\m}{\a},
$$
$\n$-a.e. $x \in \R$. This proves the Theorem.
\end{proof}

\subsection{Countable subgroups} A random walk of law $\m$ is considered non-degenerate if the semigroup generated by the support $Supp(\m)$ is in fact a subgroup of $Sol$. Theorem \ref{Cantor} permits to deduce the

\begin{corollary}\label{fingen}
Any countable subgroup $\G$ of $Sol$ not included in the horizontal plane $\R^2=\{z=0\}$ admits a non-degenerate finite first moment random walk $\m$ with harmonic measure singular on the boundary. Moreover, $\m$ can be chosen to be finitely supported when $\G$ is finitely generated.
\end{corollary}

The same result for the group $SL(d,\R)$ instead of $Sol$ was proved by Kaimanovich and Le Prince in \cite{KLeP}. 
We use the same strategy, constructing random walks that have uniformly bounded entropy and arbitrary large speed. On the other hand, we will construct in Section \ref{Pisot} random walks having arbitrary small speed and singular harmonic measure by using specific measures $\m$.

\begin{proof}
By chosing sufficiently fast decay of mass towards infinity, the countable group $\G$ admits a non-degenerate probability measure $\bar \m$ with finite entropy and finite first moment. 
When $\G$ is finitely generated, $\bar \m$ can be chosen to be finitely supported.
By assumption, $\G$ contains an element $g$ with non-zero vertical component $z(g)$. For an integer $l$, take $\m_l=\frac{1}{2}(\bar \m+\d_{g^l})$, where $\d_{g^l}$ denotes the Dirac mass at $g^l$. The entropy of the associated random walk is bounded above by $h_{\m_l} \leq H(\m_l)=\frac{1}{2} H(\bar \m)+\log 2$, and its speed is given by $\a_l=\E (\m_l)_z=\frac{1}{2} (\E(\bar\m)_z+lz(g))$. By Theorem \ref{Cantor}, the dimension of the harmonic measure is less than
$$
\dim \n_l \leq \frac{\frac{1}{2}H(\bar \m)+\log2}{|\a_l|} \underset{l\to \infty}{\longrightarrow} 0. 
$$
For $l$ large enough, we have $\dim \n_l<1$, so the harmonic measure is singular with respect to Lebesgue measure by Proposition \ref{unique-harm}.
\end{proof}

\begin{remark}\label{plusminus}
By choosing negative or positive powers of $l$ in the above proof, the support of the singular harmonic measure $\n_l$ for the random walk of law $\m_l$ on $\G$ can be chosen to be either included in $\partial^+Sol$ or included in $\partial^-Sol$. 
\end{remark}

\begin{proof}[Proof of Theorem \ref{main} (1)]
A non-abelian subgroup $\G$ of $Sol$ is not contained in the horizontal plane, so by Corollary \ref{fingen} there exists a non-degenerate finite first moment measure $\m$ with completely singular harmonic measure. There remains to prove that $\n$ has no point mass, which amounts by Proposition \ref{non-atomic} to verifying that the action of $\G$ on the boundary is non-elementary. 

Let $g$ be an element of $\G$ with non-zero vertical component. By Fact \ref{pp}, the action at a boundary is non-elementary unless for any $g'$ in $\G$, we have $p^+(g)=p^+(g')$ and $p^-(g)=p^-(g')$. But in this case, we have $g'=g^{\frac{z'}{z}}$ which belongs to $\{g^t=(tz, x(\frac{e^{-tz}-1}{e^{-z}-1}),y(\frac{e^{tz}-1}{e^{z}-1})\}_{t \in \R} $, which is a $1$-dimensional abelian Lie subgroup, so $\G$ is abelian.
\end{proof}

\begin{remark}
More precisely, the above proof and Remark \ref{plusminus} show that if the countable non-abelian group $\G$ is included in a hyper-surface $\{p^+(g)=c\}$ (resp. $\{p^-(g)=c\}$) for a constant $c$, then there exists a non-degenerate finite first moment probability measure $\m$ on $\G$ such that the harmonic measure is non-atomic singular with support included in $\partial^-Sol$ (resp. $\partial^+Sol$). If $\G$ is not included in such hypersurfaces, then we can find a non-degenerate measure $\m^+$ with non-atomic singular harmonic measure $\n^+$ on $\partial^+Sol$, as well as $\m^-$ with non-atomic singular $\n^-$ on $\partial^-Sol$.
\end{remark}

\section{Absolute continuity of harmonic measures}\label{abs}

The aim of this section is to give exemples of probability measures for which the random walk has an associated harmonic measure absolutely continuous with respect to Lebesgue measure. By symmetry, we focus on the case $\a>0$ and identify the boundary with $\R=\partial^+Sol$.

Recall that a probability measure $\m$ on $Sol$ gives independence to the $z$ and $x$ components if the projection $\pi_\ast \m$ on the $zx$-plane is a product measure $\m_z \times \m_x$. In this case, and if the support of the vertically projected measure $\m_z$ is included in a lattice $\g\Z$ in $\R$, the harmonic measure $\n$ is tightly related to the Bernoulli convolution of parameter $e^{-\g}$.

\subsection{Bernoulli convolutions}\label{bernoulli} The Bernoulli convolution $b_\l$ with parameter $\l \in ]0,1[$ is the convolution measure $(\frac{1}{2}\d_{\l^{j}}+\frac{1}{2}\d_{-\l^{j}})^{\ast j\in \N}$, where $\d_a$ denotes the Dirac mass at point $a$. In other terms, $b_\l$ is the probability distribution of the random variable $\sum_{j=0}^{\infty}x_{j}\l^{j}$, where $\{x_j\}_{j \in \N}$ is a sequence of independent variables equidistributed on the set $\{1,-1\}$.

These measures have been studied since 1930's. Simple observations show that if $\l$ belongs to $]0, \frac{1}{2}[$, the measure $b_{\l}$ is singular with respect to Lebesgue measure, since it is supported on a Cantor set, and $b_{\frac{1}{2}}$ is the Lebesgue measure itself on the interval $[-2,2]$.

The most famous question about Bernoulli convolution is to determine for which $\l$ in $]\frac{1}{2}, 1[$ the measures $b_{\l}$ are absolutely continuous or completely singular with respect to Lebesgue measure. The relevance of this question was pointed out by Erd\"os, who proved the two following results. Definition and basic facts about Pisot numbers are presented in the Appendix \ref{appendpisot}.

\begin{theorem}[Erd\"os 1939 \cite{E}]\label{erdos1}
Let $\l$ be the inverse of a Pisot number, then the Bernoulli convolution $b_\l$ is singular with respect to Lebesgue measure.
\end{theorem}

However, Pisot numbers are measure theoretically exceptional numbers, and Erd\"os proved that absolute continuity, and even existence of regular densities, hold for almost all parameters in a neighborhood of $1$.

\begin{theorem}[Erd\"os 1940 \cite{E2}]\label{erdosCk}
For any $k \in \N$, there exists $\l_{k} < 1$ such that $b_{\l}$ has a density of class $C^{k}$ for Lebesgue a.e. $\l \in [\l_{k}, 1]$.\end{theorem}

In particular, this implies that almost all Bernoulli convolutions in a left neighborhood of $1$ are absolutely continuous. This left-neighborhood is as big as one can expect as shown by Solomyak \cite{S} (see also \cite{PS} for a simple proof).

\begin{theorem}[Solomyak 1995 \cite{S}]\label{solom}
For Lebesgue a.e. $\l$ in $[\frac{1}{2},1[$, the Bernoulli convolution $b_\l=(\frac{1}{2} \d_{\l^n}+\frac{1}{2}\d_{-\l^n})^{\ast n \in \N}$ is absolutely continuous with respect to Lebesgue measure.
\end{theorem}  

It is still an open question whether Pisot numbers are the only parameters in $[\frac{1}{2},1[$ for which $b_\l$ is singular. More information about Bernoulli convolutions can be found in the expository article by Peres, Schlag and Solomyak \cite{PSS}. These three theorems permit to prove similar results in the context of harmonic measures at the boundary of $Sol$.

\subsection{Absolute continuity via Bernoulli convolutions} Recall from Section \ref{poissonSol} that the harmonic measure $\n$ is the law of the random variable $\x=\sum_{j=1}^\infty x_j e^{-S_{j-1}}$, where $X_j=(z_j,x_j,y_j)$ are independent of law $\m$, and $S_j=z_1+\dots+z_j$ are the partial sums of vertical components.

 When the measure $\pi_\ast \m$ inherited on the $zx$-plane is a product measure between the Dirac mass $\m_z=\d_{\log(\frac{1}{\l})}$ and $\m_x=\frac{1}{2}\d_1+\frac{1}{2}\d_{-1}$, we recover exactly the Bernoulli convolution of parameter $\l$. Note that this situation is degenerate in the sense that $Supp(\m)$ only generates a semi-group in $Sol$. However, we prove the following, related to Solomyak's Theorem \ref{solom}.

\begin{theorem}\label{thmsolomyak}
Assume that $\m$ has finite first moment and that $\pi_\ast \m=\m_z \times \m_x$ is a product measure between 
\begin{enumerate}
\item $\m_z=p\d_\g+(1-p)\d_{-\g}$ with $p > \frac{1}{2}$ (thus $\a=\E\m_z=(2p-1)\g>0$),
\item and $\m_x=\frac{1}{2}\d_1+\frac{1}{2}\d_{-1}$.
\end{enumerate}
Then for Lebesgue-a.e. choice of parameter $\g \in ]0,\log(2)]$, the harmonic measure $\n$ on $\R$ corresponding to the pair $(Sol, \m)$ is absolutely continuous with respect to Lebesgue measure (for any $p > \frac{1}{2}$).
\end{theorem}

Given $\o$ in $\O$, define the sequence $\z=(S_j)_{j=0}^\infty$  of integers $S_j=(z_1+\dots + z_j)/\g$. Let $\O_0$ be the set of semi-infinite path emanating from zero in $\Z$. The map $proj:\O \rightarrow \O_0$ given by $proj(\o)=\z$ describes the random walk obtained by projection on the $z$-axis. Denote $\P_0=\P \circ proj^{-1}$ the push-forward measure of $\P$ by this map.

If $\pi_\ast \m$ is a product measure, the harmonic measure can be decomposed along conditional probability measures $\{\n_\z\}_{\z \in \O_0}$ such that:
\begin{eqnarray}\label{decomposition}\n=\int_{\O_0} \n_\z d\P_0(\z), \end{eqnarray}
and $\n_\z$ is the distribution of the random variable $\x_\z=\sum_{j=1}^\infty x_j e^{-\g S_{j-1}}$, where $\z=\{S_j\}_{j=0}^{\infty}$ is fixed, and $\{x_j\}_{j=0}^{\infty}$ are independent variables of law $\m_x$. By this decomposition, Theorem \ref{thmsolomyak} is a direct consequence of Solomyak's Theorem \ref{solom}.

\begin{proof}[Proof of Theorem \ref{thmsolomyak}]
Denote $E$ the set of $\l \in [\frac{1}{2},1[$ such that $b_\l$ is absolutely continuous with respect to Lebesgue measure. It is sufficient to prove that $\n$ is absolutely continuous when $e^{-\g} \in E$.

We use decomposition (\ref{decomposition}) and note that $\n_\z$ is a convolution:
\begin{eqnarray*}\n_\z=\left(\frac{1}{2} \d_{e^{-\g S_{j-1}}}+\frac{1}{2}\d_{-e^{-\g S_{j-1}}}\right)^{\ast j \in \N}.\end{eqnarray*}
It is sufficient to check that $\n_\z$ is absolutely continuous for $\P_0$-a.e. $\z$.

For any real numbers $\l_1,\l_2$, we have commutation of convolutions: 
$$(\frac{1}{2} \d_{\l_1}+\frac{1}{2}\d_{-\l_1})\ast (\frac{1}{2} \d_{\l_2}+\frac{1}{2}\d_{-\l_2})= (\frac{1}{2} \d_{\l_2}+\frac{1}{2}\d_{-\l_2})\ast (\frac{1}{2} \d_{\l_1}+\frac{1}{2}\d_{-\l_1}).$$
This permits to rewrite:
\begin{eqnarray}\label{zetaconvolution}\n_\z=\left(\left (\frac{1}{2} \d_{e^{-\g k}}+\frac{1}{2}\d_{-e^{-\g k}}\right)^{\ast n(\z,k)}\right)^{\ast k \in \Z}, \end{eqnarray}
where $n(\z,k)=\#\{j \in \N| S_{j-1}=k\} $ is the time spent by the vertical random walk at position $k$. Almost surely with respect to $\P_0$, it satisfies $n(\z,k) \geq1$ for all $k\geq 0$ and there exists $k_0$ with $n(\z,k)=0$ for all $k \leq k_0$ $\P_0$-almost surely (see Appendix \ref{probafacts}). 

Using commutation once more, for $\P_0$-a.e. choice of $\z$, we can factorize $\n_\z=b_{e^{-\g}} \ast \n_{aux}$ for some auxiliary measure $\n_{aux}$. The measure $b_{e^{-\g}}$ is absolutely continuous with respect to Lebesgue because $e^{-\g}$ belongs to $E$. By convolution, $\n_\z$ is also absolutely continuous with respect to Lebesgue measure for $\P_0$-a.e. $\z$. 
Thus by (\ref{decomposition}), $\n$ is absolutely continuous.
\end{proof}

\begin{remark}
As observed by Kahane \cite{Ka} (see also Section 6 in \cite{PSS}), the Hausdorff dimension of the set of parameters $\l$ in an interval $[\l_0,1]$ with $b_\l$ singular tends to zero as $\l_0$ approaches $1$. By the above proof, this guarantees a similar result in our setting, namely that the Hausdorff dimension of the set of parameters $\g$ in an interval $[0,\g_0]$ such that the conclusion of Theorem \ref{thmsolomyak} does not hold tends to $0$ when $\g_0$ approaches $0$. 
A recent notable result by Shmerkin states that the Hausdorff dimension of the set of parameters $\l$ in the interval $[1/2,1]$ with $b_\l$ singular is in fact zero \cite{Sh}. Again, a similar result in our setting holds for the set of parameter $\g$ in the interval $]0, \log(2)]$.
\end{remark}

\begin{remark}\label{remPS2} The particular choice of measures $\m_z$ and $\m_x$ in the hypothesis of Theorem \ref{thmsolomyak} is due to the necessity to apply results about Bernoulli convolution.  By Theorem 1.3 and Corollary 5.2 in \cite{PS2},
Theorem \ref{thmsolomyak} can be generalized to the case $\m_x=\sum_{i=1}^m p_i \d_{d_i}$, for which the harmonic measure $\n$ is absolutely continuous for Lebesgue-a.e. choice of parameter $\log(1+\sqrt{b}) \leq \g \leq H(\m_x)$, where $b=\sup\{|\frac{d_i-d_j}{d_k-d_l}|, 1 \le i, j, k, l \le m, d_{k} \neq d_{l}\}$. One may naturally ask about generalization to other measures, but already for Bernoulli convolution, this seems to be a difficult task.
\end{remark}

 \subsection{Densities in the class $C^k$} Erd\"os' Theorem \ref{erdosCk} permits to construct random walks on $Sol$ with harmonic measure admitting a density function of class $C^k$.

\begin{theorem}\label{Ck}
Let $\m$ satisfy the hypothesis of Theorem \ref{thmsolomyak}. For any $k \in \N$, there exists $\g_k>0$ such that for Lebesgue-a.e. choice of parameter $\g \in ]0,\g_k[$, the harmonic measure $\n$ admits a density of class $C^k$.
\end{theorem}

The regularity of the density function of a distribution $\n$ is classicaly related to the asymptotic decay of its Fourier transform $\hat \n (t)=\int_\R e^{i t \x}d\n(\x)$, where $i=\sqrt{-1}$. For instance:

\begin{lemma}[Riemann-Lebesgue Lemma]\label{RL}
If the measure $\n$ is absolutely continuous with respect to Lebesgue measure, then the Fourier transform $\hat \n$ is continuous and $\n(t) \rightarrow 0$ as $|t|$ tends to infinity.
\end{lemma}

Though rarely stated in this form, the following lemma underlies the well-known fact that the Fourier transform maps the Schwartz space into itself.

\begin{lemma}\label{schwartz}
If the measure $\n$ has a density in the class $C^k$, then $\hat\n (t)=o(|t|^{-k})$.
Conversely, if $\hat\n (t)=O(|t|^{-k})$, then the measure $\n$ has a density in the class $C^{k-2}$.
\end{lemma}

Roughly, the first assertion holds because differentiating $k$ times a function multiplies its Fourier transform by $t^k$. The second hypothesis ensures that $t^{k-2}\hat\n(t)$ is integrable, hence has continuous Fourier transform tending to zero at infinity. This is also the case of the $(k-2)$nd derivative of the density of $\n$ because the Fourier transform is essentially an involution. We refer to Section II-29 in \cite{D} for detailed statements.

\begin{proof}[Proof of Theorem \ref{Ck}]
Following \cite{E2}, we consider the Fourier transform $\hat \n (t)=\int_\R e^{i t \x}d\n(\x)$ of the harmonic measure $\n$. Using decomposition (\ref{decomposition}), it is given by
\begin{eqnarray}\label{dechat}\hat \n (t) =\int_{\O_0} \widehat{\n_\z}(t)d\P_0(\z). \end{eqnarray}

The Fourier transform of the measure $\n_\z$ conditionned by $\z=\{S_j\}_{j=0}^{\infty}$, described above in (\ref{zetaconvolution}) as a convolution, is computed as a product, using the Fourier transform $\cos(ta)$ of the measure $\frac{1}{2}(\d_a+\d_{-a})$. We get
\begin{eqnarray}\label{fourierdec}
\widehat{\n_\z} (t) = \prod_{k=-\infty}^\infty  \cos(t e^{-\g k})^{n(\z,k)},
\end{eqnarray}
where the last line is obtained by setting $n(\z,k)=\#\{j \in \N| S_{j-1}=k\}$, which satisfes $n(\z,k) \geq 1$ for $k\geq 0$ and $\P_0$-a.e. $\z$ (see Section \ref{probafacts}).

On the other hand, the Fourier transform of the Bernoulli convolution of parameter $\l=e^{-\g}$ is given by $\widehat{b_{e^{-\g}}}(t)= \prod_{k=0}^\infty  \cos(t e^{-\g k})$. This shows that $|\widehat{\n_\z} (t)| \leq |\widehat{b_{e^{-\g}}}(t)|$ for $\P_0$-a.e. sample path $\z$, so $|\hat{\n}(t)| \leq |\widehat{b_{e^{-\g}}}(t)|$.

Now take $k \in \N$ and let $\l_{k+2}$ and $E \subset [\l_{k+2},1]$ of full measure be given by Theorem \ref{erdosCk} such that $b_\l$ has a density in the class $C^{k+2}$ for all $\l$ in $E$. 

Finally set $\g_k=-\log \l_{k+2}$. If $\g$ belongs to the set $-\log E \subset ]0,\g_k[$ of full measure, then by Lemma \ref{schwartz}, the Fourier transform $\widehat{b_{e^{-\g}}}$ is $o(|t|^{-k-2})$, as well as $\hat \n(t)$, so $\n$ admits a density in the class $C^k$.
\end{proof}

\begin{remark} Stated in these forms, Theorem \ref{thmsolomyak} and Theorem \ref{Ck} do not provide explicit description of random walks $\m$ with regular harmonic measure, but only existence for almost all values of parameters. Explicit measures can be obtained by the same proof as above, using the following:

\begin{theorem}[Wintner 1935 \cite{W}]
The Bernoulli convolution $b_\l$ with parameter $\l=(\frac{1}{2})^\frac{1}{m}$ admits a density function in the class $C^{m-2}$.
\end{theorem}

In particular, a measure satisfying the hypothesis of Theorem \ref{thmsolomyak} with $\g=\frac{\log 2}{k+4}$ admits a density function of class $C^k$.
\end{remark}

\section{The case of ``Pisot" vertical lattices}\label{Pisot}

Inspired by Erd\"os Theorem \ref{erdos1}, we focus on the case where the vertical measure $\m_z$ has support in a lattice $\g\Z$ where $e^{\g}$ is a Pisot number. It permits to construct random walks on $Sol$ with arbitrarily small speed, but harmonic measures which are singular with respect to Lebesgue measure. This hypothesis, which may seem odd at first sight, is necessarily satisfied if the support of $\m$ generates a lattice in $Sol$.

\subsection{Singular harmonic measures with small speed} We prove singularity of the harmonic measure in the case where $\m$ gives independence to $\m_z$ and $\m_x$, and both these measures are supported on a lattice. The following theorem generalizes Erd\"os Theorem \ref{erdos1}, which corresponds to the (degenerate) case $\m_z=\d_{\g}$ below.

\begin{theorem}\label{thmerdos}
Assume that $\m$ has finite first moment and that $\pi_\ast \m=\m_z \times \m_x$ is a product measure between 
\begin{enumerate}
\item $\m_z$ such that $\E\m_z=\a>0$ and with support $Supp(\m_z) \subset \g \Z$ for a real number $\g$ such that $e^\g$ is a Pisot number, 
\item and $\m_x=q_1 \d_1+q_1\d_{-1}+q_0\d_0$ with $q_0>\frac{1}{2}$.
\end{enumerate}
Then the harmonic measure $\n$ on $\R$ corresponding to the pair $(Sol, \m)$ is singular with respect to Lebesgue measure.
\end{theorem}

Theorem \ref{main} (3) follows from Theorem \ref{thmerdos}, Proposition \ref{unique-harm} and \ref{non-atomic}, since $\m$ can be chosen to be a finitely supported, non-degenerate measure such that the group generated by $Supp(\m)$ acts non-elementary on $\R$, keeping conditions (1) and (2) in Theorem \ref{thmerdos}.

\begin{proof}[Proof of Theorem \ref{thmerdos}]
By Riemann-Lebesgue Lemma \ref{RL}, it is sufficient to prove that $\hat \n(t)$ does not tend to $0$ as $t$ tends to $\infty$. To ease notations, we write $\b=e^{-\g}$ and note that $\b^{-1}$ is a Pisot number.

As in the previous proof, we use (\ref{decomposition})-(\ref{fourierdec}) to compute the Fourier transform. The law of $x_j\b^k$ is $q_1(\d_{\b^k} + \d_{-\b^k})+q_0$, with Fourier transform $2q_1\cos(t\b^k)+q_0$, so (\ref{fourierdec}) becomes
$$\widehat{\n_\z} (t)=\prod_{k=-\infty}^\infty (2q_{1} \cos(t \b^k)+q_{0})^{n(\z,k)},$$
with the notation of Section \ref{probafacts}. All the terms in the product are $\leq 1$ and greater than $-2q_{1}+q_{0}>0$ because $q_0>\frac{1}{2}$ and $2q_1+q_0=1$. By (\ref{dechat}), we get:
\begin{eqnarray}\label{nuhat}\hat \n (t) = \int_{\O_0} \prod_{k=-\infty}^\infty (2q_{1} \cos(t \b^k)+q_{0})^{n(\z,k)} d\P_0(\z).\end{eqnarray}

By the Jensen inequality and the Fubini theorem,
\begin{eqnarray*}
\log \hat \n (t) &\geq &  \int_{\O_0} \log \left( \prod_{k=-\infty}^\infty (2q_1 \cos(t\b^k)+q_0)^{n(\z,k)} \right) d\P_0(\z) \\
 &=& \int_{\O_0} \sum_{k=-\infty}^\infty n(\z,k) \log (2q_1 \cos(t\b^k)+q_0) d\P_0(\z) \\
 &=& \sum_{k=-\infty}^\infty \int_{\O_0} n(\z,k) d\P_0(\z) \log (2q_1 \cos(t\b^k)+q_0).
\end{eqnarray*}
Lemma \ref{lemmaproba} gives $0< \int_{\O_0}n(\z,k) d\P_0(\z)=\E n(\z,k) \leq M <\infty$. As $\log(2q_1\cos(t\b^k)+q_0)<0$, we get:
\begin{eqnarray*}
\log \hat \n (t) \geq  M \sum_{k=-\infty}^\infty  \log (2q_1 \cos(t\b^k)+q_0) = M \log \left( \prod_{k=-\infty}^\infty (2q_1\cos(t\b^k)+q_0) \right),
\end{eqnarray*}
thus:
\begin{eqnarray}\label{erdos} \hat \n(t)^\frac{1}{M} \geq \prod_{k=-\infty}^\infty (2q_1 \cos(t\b^k)+q_0). \end{eqnarray}
The right-hand side is almost the Fourier transform of the Bernoulli convolution with parameter the inverse of a Pisot number, so the remainder of our proof follows Erd\"os \cite{E}. For any integer $l$,  set $t_l=2\pi \b^l$. We prove that there exists $c>0$ such that $\hat \n(t_l)^\frac{1}{M} \geq c$ for all $l$ in $\Z$.

By Lemma \ref{lemmapisot}, there exists $\th<1$ and $L$ such that: 
\begin{eqnarray*}
\hat \n(t_l)^\frac{1}{M} &\geq & \prod_{k=-\infty}^\infty (2q_1 \cos(2\pi \b^{l+k})+q_0) \\
& \geq & \prod_{|k|\geq L}(2q_1(1-\th^{|k|})+q_0) \prod_{|k|<L} (2q_1\cos(2\pi \b^k)+q_0)=c>0,
\end{eqnarray*}
where $l$ disappears by translation invariance. The first product $\prod_{|k|\geq L}(1-2q_1\th^{|k|})>0$ is non-zero by exponential decay and the second has finitely many positive factors.
\end{proof}

\begin{remark}
Theorem \ref{thmerdos} is still true for $\m_x$ a symmetric measure on $\Z$ where $q_0 >\frac{1}{2}$ and the sequence $\m_x(r)=\m_x(-r)=q_r$ has a finite $\eta$-moment for some $\eta>1$. 

More precisely, let $\b^{-1}$ be a Pisot number and consider $\th<1$ from Lemma \ref{lemmapisot}. Under the moment condition, there exists $1<\s<\th^{-1}$ such that 
$\sum_{k=1}^\infty \sum_{r \geq \s^k} q_r <\infty$. We deduce that the harmonic measure $\n$ is singular. Indeed, (\ref{erdos}) becomes:
$$\hat \n(t_l)^\frac{1}{M} \geq \prod_{k=-\infty}^\infty (q_0+\sum_{r=1}^\infty 2q_r \cos(2\pi r \b^k) ).$$
By Lemma \ref{lemmapisot}, there exists $L'$ such that for $|k| \geq L'$ and $1 \leq r \leq \s^{|k|}$, we have $|\cos(2\pi \b^kr)-1| \leq r\th^{|k|} \leq (\s\th)^{|k|}$. Then:
\begin{eqnarray*}
\hat \n(t_l)^\frac{1}{M} &\geq & \prod_{|k|\geq L'} (q_0+\sum_{r=1}^{\s^{|k|}}2q_r (1-(\s\th)^{|k|})-\sum_{r>\s^{|k|}}2q_r ) \prod_{|k|< L'} (q_0+\sum_{r=1}^\infty 2q_r \cos(2\pi r \b^k) ) \\
& \geq & \prod_{|k|\geq L'}(1-2(\s\th)^{|k|}-4\sum_{r>\s^{|k|}}q_r) \prod_{|k|< L'} (q_0+\sum_{r=1}^\infty 2q_r \cos(2\pi r \b^k) )=c>0.
\end{eqnarray*}
The left-side product converges by the assumption on the decay of $(q_r)_{r}$ and because $\s\th<1$.
\end{remark}

\subsection{Cocompact lattices}\label{coc}

For any matrix $T $ in $SL(2, \Z)$ with trace satisfying $Tr(T)>2$, denote $\G_T$ the semi-direct product $\Z\ltimes_{T} \Z^{2}$, where $r \in \Z$ acts on $(p,q) \in \Z^2$ by $r.(p,q)=T^r(p,q)$. The abstract group $\G_T$ can be realized as a cocompact lattice in $Sol$. 

Indeed, let $0<e^{-\g}<1$ and $e^\g$ be the the eigenvalues of $T$. Note that they are the roots of $X^2-Tr(T)X+1$, so $e^\g$ is a Pisot number. By change of basis $B$, we diagonalize $T$ as
$$
BTB^{-1}=
\left(\begin{array}{cc}
e^{-\g} & 0 \\ 0 & e^\g
\end{array}\right).
$$
The homomorphism $\p:\G_T \rightarrow Sol$ given by $\p(r,p,q)=(r\g,B(p,q))$ is injective. Its image $\p(\G_T)$ is a cocompact lattice since the quotient space is a torus fiber bundle over the circle. 

In fact, any cocompact lattice in $Sol$ has this form by \cite{MR} or \cite{MS}. Moreover, any finitely generated group quasi-isometric to $Sol$ is virtually (up to taking finite index subgroup) a cocompact lattice in $Sol$ by \cite{EFW}.

A random walk on the finitely generated group $\G_T$ can be viewed in $Sol$ via the above homomorphism $\p$. In this case, the visual boundary $\partial Sol$ with the harmonic measure $\n$ from Theorem \ref{boundary} is actually the Poisson boundary of $\G_T$ by Kaimanovich (\cite{Ksolv},\cite{Khyperbolic}). By Corollary \ref{fingen}, there exists a non-degenerate measure $\m$ on $\G_T$ such that the harmonic measure $\n$ on the boundary is singular with respect to Lebesgue measure. However, we have not answered the following:
\begin{question}
Does there exist a random walk $\m$ (with finite support) on a cocompact lattice of $Sol$ such that the harmonic measure $\n$ on the boundary is absolutely continuous with respect to Lebesgue ?
\end{question}
Neither Theorem \ref{thmsolomyak} nor its extension in Remark \ref{remPS2} applies to this question because the set of Pisot numbers has zero Lebesgue measure. Theorem \ref{thmerdos} either, since for a non-degenerate random walk on $\p(\G_T)$, the projected measure $\m_x$ generates a dense subgroup of $\R$, rather than a lattice.

\section{Appendix: Classical Facts}\label{App}

\subsection{About Pisot numbers} \label{appendpisot} A real number $\a>1$ is a Pisot number if there exists a polynomial $P(X)=X^r+a_{r-1}X^{r-1}+\dots+a_0$ with integer coefficients and roots $\{\a,\a_2,\dots ,\a_r\}$ satisfying $|\a_s|<1$ for all $2 \leq s \leq r$. For instance, the Golden ratio is a Pisot number, root of $X^2-X-1$. These numbers are interesting because their powers are very close to being integers.

\begin{fact}\label{fact}
If $\a$ is a Pisot number, there exists $\tilde \th<1$ such that $dist(\a^k,\Z) \leq \tilde{\th}^k$ for all $k \in \N$.
\end{fact}

\begin{proof}
For each $k$, the quantity $\a^k+\a_2^k+\dots+\a_r^k$ is a symmetric polynomial in the roots of $P$, which can be expressed as a polynomial expression of the coefficients of $P$, hence is an integer. This shows $dist(\a^k,\Z) \leq \a_2^k+\dots+\a_r^k \leq (r-1)\d^k$ where $\d=\max_{s=2,\dots,r}|\a_s|<1$.
\end{proof}

This Fact \ref{fact} will be more handy to us in the following form.

\begin{lemma}\label{lemmapisot}
For any parameter $0<\b<1$ such that $\frac{1}{\b}$ is Pisot, there exists $\th<1$ and $L$ such that $|k|\geq L$ implies $|\cos(2\pi \b^k)-1|\leq \th^{|k|}$. 
\end{lemma}

\begin{proof}
For $k \geq 0$ large enough, $|\cos(2\pi\b^k)-1|\leq 2\pi \b^k$ as $\b<1$. On the other hand:
$$dist(2\pi \b^{-k},2\pi \Z)=2\pi dist(\left(\frac{1}{\b}\right)^k,\Z) \leq 2\pi \tilde{\th}^k $$ 
by Fact \ref{fact}, hence $|\cos(2\pi\b^{-k})-1|\leq 2\pi \tilde{\th}^k$. Take $\th>\max\{\b,\tilde \th\}$.
\end{proof}

\subsection{About random walks on the integers}\label{probafacts} Let $\m_z$ be a probability measure on $\Z$ of mean $\a=\E \m_z$, and consider the associated random walk $S_j=z_1+\dots+z_j$, where $z_i$ are independent $\m_z$-distributed integers. We denote $\P_0$ the inherited measure on the space $\O_0$ of paths in $\Z$ emanating from zero. 

For a sample path $\z=\{S_j\}_{j=0}^\infty$ and an integer $k$, denote $n(\z,k)=\#\{j \in \N|S_j=k\}$ the amount of time spent in position $k$ by the random walk. For each fixed $k$, the function $n(.,k):\O_0 \rightarrow \N$ is measurable. 

By the law of large numbers, $S_j \sim \a j$ almost surely. Therefore if $\a \neq 0$ for each integer $k$ and for $\P_0$ almost every path $\z$, we have $n(\z,k)<\infty$. Moreover, if $\a>0$ (respectively $\a<0$) there almost surely exists an integer $k_0$, depending on the sample $\z$ such that $n(\z,k)=0$ for all $k\leq k_0$ (resp. $k \geq k_0$).

The average time spent in position $k$ is estimated in the following lemma.

\begin{lemma}\label{lemmaproba}
If $\E\m_z=\a\neq 0$, then $\E n(\z,0)=M<\infty$. Moreover, if $\a>0$ (resp. $\a<0$), we have $\E n(\z,k)=M$ for all $k \geq 0$ (resp. $k \leq 0$) and $\E n(\z,k)\leq M$ for all $k \leq 0$ (resp. $k \geq 0$).
\end{lemma}

\begin{proof}
By definition, $\E n(\z,k)=\sum_{m=1}^\infty m \P_0[n(\z,k)=m]$. Conditionning by the first hitting time $\t_k=\min \{j \geq 0 |S_j=k\}$, we get 
$$\P_0[n(\z,k)=m]=\P_0[n(\z,k)=m|\t_k<\infty]\P_0[\t_k<\infty]. $$
By strong Markov property and translation invariance, $\P_0[n(\z,k)=m|\t_k<\infty]=\P_0[n(\z,0)=m]$. We deduce $\E n(\z,k)=\P_0[\t_k<\infty]\E n(\z,0)$. For $\a>0$ (resp. $\a<0$), we have $\P_0[\t_k<\infty]=1$ for $k\geq 0$ (resp. $k \leq 0$). This proves the second part.

To get the first part, consider the first return time $\r_0=\min\{j \geq 1|S_j=0\}$. As the random walk is transient, this return time is infinite with positive probability $\P_0[\r_0=\infty]=1-p>0$ and $\P_0[\r_0<\infty]=p<1$. 

These equalities provide the case $m=1$ in the statement $\P_0[n(\z,0)=m]=(1-p)p^{m-1}$ and $\P_0[n(\z,0)>m]=p^m$, which we prove by induction, using the strong Markov property:
$$\begin{array}{l}
\P_0[n(\z,0)=m+1] =\P_0[n(\z,0)=m+1|n(\z,0)>m]\P_0[n(\z,0)>m]=(1-p)p^m, \\ 
\P_0[n(\z,0)>m+1] = \P_0[n(\z,0)>m+1|n(\z,0)>m]\P_0[n(\z,0)>m]=p^{m+1}. \end{array}
$$
In conclusion, $M=\E n(\z,0)=(1-p)\sum_{m=1}^\infty mp^{m-1}=\frac{1}{1-p}\geq 1.$
\end{proof}

\textbf{Acknowledgements.} We wish to thank Koji Yano who pointed out the connection to Bernoulli convolutions, Philippe Castillon, Vadim A. Kaimanovich, and Takefumi Kondo for valuable comments, Yuval Peres for notifying us of \cite{Sh} and helpful comments, Boris Solomyak for informing us of \cite{BPS} and \cite{Bo}, Wolfgang Woess for a discussion about boundaries of $Sol$, and for sending us the preliminary version of \cite{Wo}, as well as the anonymous referee for helpful comments.
R.T. is supported by the JSPS Grant-in-Aid for Research Activity Start-up (No. 24840002).

\textsc{\newline J\'er\'emie Brieussel \newline Universit\'e Montpellier 2 \newline Place E. Bataillon cc 051 \newline 34095 Montpellier, France} \newline
\textit{E-mail address:} jeremie.brieussel@univ-montp2.fr

\textsc{\newline Ryokichi Tanaka \newline Tohoku University \newline Aoba-ku, Aramaki, Aoba 6-3 \newline 980-8578 Sendai, Japan} \newline
\textit{E-mail address:} r-tanaka@math.tohoku.ac.jp

\end{document}